\documentclass{article}
\usepackage[utf8]{inputenc}

\usepackage{amsmath}
\usepackage{amsthm}
\usepackage{amssymb}
\usepackage[backref=page,colorlinks=true]{hyperref}

\usepackage{color}

\newcommand{\inj}{\operatorname{inj}}

\newcommand{\conj}{\operatorname{conj}}

\newcommand{\vol}{\operatorname{vol}}

\newcommand{\diam}{\operatorname{diam}}

\newcommand{\Per}{\operatorname{Per}}

\allowdisplaybreaks

\newtheorem{theorem}{Theorem}[]
\newtheorem{lemma}[theorem]{Lemma}
\newtheorem{proposition}[theorem]{Proposition}

\theoremstyle{definition}
\newtheorem{definition}[theorem]{Definition}

\theoremstyle{remark}

\allowdisplaybreaks
\numberwithin{equation}{section}
\numberwithin{theorem}{section}

\title{Characterizing unit spheres in Euclidean spaces via reach and volume}

\author{Mark Iwen\thanks{Michigan State University, Department of Mathematics, and the Department of Computational Mathematics, Science and Engineering (CMSE), \texttt{markiwen@math.msu.edu}.  Supported in part by NSF DMS 1912706 and by NSF DMS 2106472.}, 
Benjamin Schmidt\thanks{Michigan State Univeristy, Department of Mathematics, \texttt{schmidt@math.msu.edu}.  Supported in part by a Simons Collaboration Grant 712405.},
Arman Tavakoli\thanks{Michigan State University, Department of Mathematics, \texttt{tavakol4@msu.edu}}
}

\begin{document}

\maketitle

\begin{abstract}
Let $M$ be a smooth, connected, compact submanifold of $\mathbb{R}^n$ without boundary and of dimension $k\geq 2$.  Let $\mathbb{S}^k \subset \mathbb{R}^{k+1}\subset \mathbb{R}^n$ denote the $k$-dimesnional unit sphere.  We show if $M$ has reach equal to one, then its volume satisfies $\vol(M)\geq \vol(\mathbb{S}^k)$ with equality holding only if $M$ is congruent to $\mathbb{S}^k$.  
\end{abstract}

\section{Introduction}
Let $M$ be a smooth, connected, and closed $k$-dimensional submanifold of $\mathbb{R}^n$.  Let $\rho:\mathbb{R}^n \times \mathbb{R}^n \rightarrow \mathbb{R}$ denote the Euclidean metric and $\rho_M:\mathbb{R}^n \rightarrow \mathbb{R}$ the distance function to $M$ defined by $\rho_M(x)=\rho(x,M)$.  The \textit{reach} of $M$ is the positive real number $\tau(M)$ defined by $$\tau(M)=\sup \left\{ t > 0 \,\big| \, \text{Each point in}\, \rho_M^{-1}([0,t))\, \text{ has a unique closest point in}\, M \right\}.$$ The reach of a subset of Euclidean space was introduced by Federer in the influential paper \cite{fe}. The reach of a closed submanifold as above equals its normal injectivity radius. The $k$-dimensional unit sphere $$\mathbb{S}^k=\left\{x=(x_1,\ldots,x_n)\in \mathbb{R}^n\,\big|\, \|x\|=1\,\text{and}\,\, x_{i}=0\,\, \text{if}\,\, k+2\leq i\leq n\right\}$$ has $\tau(\mathbb{S}^k)=1$. The scale invariant ratio $$\vol(M)/\tau(M)^k$$ arises in estimates for the number of metric balls in $\mathbb{R}^n$ needed to cover $M$ when the balls are required to be centered in $M$ and to have equal radii (see, e.g., \cite{baraniuk2009random, efwa, iwscta1, iwscta2}). These estimates have applications in compressive sensing and mathematical data science where they are combined with probabilistic methods to estimate, e.g., the smallest dimension $m<n$ such that $M$, equipped with the restriction of the metric $\rho$, admits a bilipshitz map to $\mathbb{R}^m$ with bilipshitz constants close to $1$. In \cite[Proposition 4.2]{iwscta2}, G\"unther's volume comparison theorem and an injectivity radius estimate were applied to establish the inequality $$\vol(M)/\tau(M)^k \geq \vol(\mathbb{S}^k)/\tau(\mathbb{S}^k)^k=\vol(\mathbb{S}^k).$$ Herein, we show that equality holds only for spheres.

\begin{theorem}\label{main theorem}
Let $M$ be a smooth, connected, and closed $k$-dimensional submanifold of $\mathbb{R}^n$ with $k\geq 2$.  If $\tau(M)=1$, then $\vol(M) \geq \vol(\mathbb{S}^k)$ with equality only if there exists an isometry $I$ of $\mathbb{R}^n$ such that $I(\mathbb{S}^k)=M$.  
\end{theorem}     

The proof consists of two main steps.  The first step is to show that $M$, with the induced Riemannian metric, is isometric to $\mathbb{S}^k$.  The second step is to show that $M$ is embedded in $\mathbb{R}^n$ as an isometric image of the standard $\mathbb{S}^k$.  In the first step, the hypothesis $\tau(M)=1$ is used to bound the injectivity radius of $M$ below by $\pi$ after which Berger's sharp isoembolic inequality \cite{berger} is used to show $M$ and $\mathbb{S}^k$ are isometric.  Hong's theorem \cite{Ho} reduces the second step to showing that each geodesic in $M$, a closed geodesic of length $2\pi$, is the image of the standard $\mathbb{S}^1$ under some isometry of $\mathbb{R}^n$.  Finally, the solution of a well known puzzle \cite{Pe,Wi} concerning closed curves in spheres applies to show that the geodesics of $M$ are indeed unit circles in $\mathbb{R}^n$. 

\section{Preliminaries}\label{preliminary}

\subsection{ Unit circles in \texorpdfstring{$\mathbb{R}^n$}{Rn} and closed curves in \texorpdfstring{$\mathbb{S}^{n-1}$}{S(n-1)}}

A subset $S$ of $\mathbb{R}^n$ is a \textit{unit circle} if there exists an isometry $I$ of $\mathbb{R}^n$ such that $S=I(\mathbb{S}^1)$. A parametric characterization of unit circles is given in Proposition \ref{round} below.  It is based on a solution to a puzzle about closed curves in $\mathbb{S}^{n-1}$ appearing in \cite{Pe,Wi}.  

To state the puzzle, equip $\mathbb{S}^{n-1}$ with the Riemannian metric induced from $\mathbb{R}^n$.  The geodesics in $\mathbb{S}^{n-1}$ are unit circles with center of mass the origin.  The geodesic distance function $d:\mathbb{S}^{n-1} \times \mathbb{S}^{n-1} \rightarrow \mathbb{R}$ is given by $d(p,q)=\arccos(\langle p,q\rangle)$.  Each $m\in \mathbb{S}^{n-1}$ is the pole of a unique hemisphere $$H_m=\left\{v \in \mathbb{S}^{n-1}\,\vert\, \langle v, m\rangle>0 \right\}= \left\{v\in \mathbb{S}^{n-1}\, \vert\, d(v,m)<\frac{\pi}{2} \right\}$$ and this hemisphere is bounded by an equitorial subsphere  $$E_m=\left\{v \in \mathbb{S}^{n-1}\, \vert\, \langle v,m\rangle=0 \right\}= \left\{v\in \mathbb{S}^{n-1}\,\vert\, d(v,m)=\frac{\pi}{2} \right\}.$$  Consider the following puzzle: \textit{Prove if a closed curve on the unit sphere has length less than $2\pi$, then it is contained in some hemisphere.}

We present an elegant solution to this puzzle taken from \cite{Pe, Wi} as Lemma \ref{long} below.  Understanding the boundary case of the puzzle leads to the desired parametric characterization of unit circles in Proposition \ref{round}.  For our purposes, it is sufficient to work with curves in $\mathbb{R}^n$ admitting smooth parameterizations as described in the following definitions.

\begin{definition}
A \textit{parameterization} is a smooth map $x:\mathbb{R} \rightarrow \mathbb{R}^n$.  A paramaterization $x$ has \textit{unit speed} if $\|x'(t)\|=1$ for all $t \in \mathbb{R}$.    A \textit{curve} in $\mathbb{R}^n$ is a subset $\Gamma$ of $\mathbb{R}^n$ for which there exists a parameterization $x$ with $x(\mathbb{R})=\Gamma$.    
\end{definition}

Curves in $\mathbb{R}^n$ may not be the image of a unit speed parameterization.  For instance, each point in $\mathbb{R}^n$ is a curve as the image of a constant parameterization.  Such curves, called \textit{point curves}, do not admit a unit speed parameterization. Given a parameterization $x$, let $$\Per(x)=\{T \in \mathbb{R}\, \vert\, \forall t\in \mathbb{R},\,x(t)=x(t+T)\}.$$  $\Per(x)$ is a closed subgroup of $(\mathbb{R},+)$.  Therefore $\Per(x)=\mathbb{R}$, $\Per(x)=\{0\}$, or there exists $P>0$ such that $\Per(x)=P\cdot\mathbb{Z}$. $\Per(x)=\mathbb{R}$ if and only if $x(\mathbb{R})$ is a point curve.  

\begin{definition}
A paramaterization $x:\mathbb{R} \rightarrow \mathbb{R}^n$ is \textit{periodic} with period $P>0$ if $\Per(x)=P\cdot \mathbb{Z}$.  A curve $\Gamma$ is \textit{closed} if there exists a periodic paramaterization $x$ with $x(\mathbb{R})=\Gamma$. 
\end{definition}

The next Lemma is a special case of the possibly intuitive assertion that a parameterized curve $x(t)$ cannot have closed image if there is a one-dimensional subspace $L$ of $\mathbb{R}^n$ such that the velocity vector $x'(t)$ projects to a nonzero vector in $L$ for each $t \in \mathbb{R}.$

\begin{lemma}\label{derivnotinhemi}
Let $C:\mathbb{R} \rightarrow \mathbb{R}^n$ be a unit speed periodic parameterization and let $c:\mathbb{R} \rightarrow \mathbb{S}^{n-1}$ be the parametrization defined by $c(t)=C'(t)$ for each $t \in \mathbb{R}$.  For each $m \in \mathbb{S}^{n-1}$, there exists $s \in \mathbb{R}$ such that $c(s) \in E_m$. 
\end{lemma}

\begin{proof}
Define $f:\mathbb{R} \rightarrow \mathbb{R}$ by $f(t)=\langle C(t), m\rangle$. Let $P>0$ be the period of $C$.  Then $$\int_{t=0}^{t=P} \langle c(t),m \rangle\, dt=\int_{t=0}^{t=P} f'(t)\,dt=f(P)-f(0)=\langle C(P)-C(0),m\rangle=0,$$ from which the Lemma follows. 
\end{proof}

If $x:\mathbb{R} \rightarrow \mathbb{R}^n$ is a parameterization and if $I \subset \mathbb{R}$ is a bounded interval, then the pair $(x,I)$ has a length defined by the familiar formula $$L(x,I)=\int_{I} \|x'(t)\|\,dt.$$  
Length is determined by the image $x(I)$. We record the following special case without proof.

\begin{lemma}\label{lengthdefined}
Let $\Gamma$ be a closed curve in $\mathbb{R}^n$.  Suppose that $x$ and $y$ are periodic parameterizations with $x(\mathbb{R})=\Gamma=y(\mathbb{R})$.  If $P_x$ and $P_y$ denote the periods of $x$ and $y$, then $L(x,[0,P_x])=L(y,[0,P_y]).$
\end{lemma}

\begin{definition}
If $\Gamma$ is a closed curve, then its \textit{length} is defined as the common value of the lengths appearing in Lemma \ref{lengthdefined}.
\end{definition}

We now present a solution to the puzzle above.  The solution appears in the puzzle books \cite{Pe,Wi}.  We include it here since the line of reasoning appears again in the proof of Proposition \ref{round} below.  

\begin{lemma}\label{long}
Let $\Gamma \subset \mathbb{S}^{n-1}$ be a closed curve having the property that for each $m\in \mathbb{S}^{n-1}$, $\Gamma \cap E_m\neq \emptyset$. If $L$ is the length of $\Gamma$, then $L\geq 2\pi$. 
\end{lemma}

\begin{proof}
Let $c:\mathbb{R} \rightarrow \mathbb{S}^{n-1}$ be a periodic parameterization with $c(\mathbb{R})=\Gamma$.  Let $P_c>0$ denote the period of $c$.  There exists $t_c \in (0,P_c)$ such that $$L(c,[0,t_c])=\frac{L}{2}=L(c,[t_c,P_c]).$$  Let $p=c(0)$ and $q=c(t_c)$ and note that $$d(p,q) \leq L(c,[0,t_c])= \frac{L}{2}.$$ Let $\gamma:[0,d(p,q)]\rightarrow \mathbb{S}^{n-1}$ be a unit speed minimizing geodesic joining $p$ to $q$. 

We argue by contradiction.  If $L<2\pi$, then \begin{equation}\label{useme} d(p,q) \leq \frac{L}{2} < \pi.\end{equation}   Letting $m=\gamma(\frac{d(p,q)}{2})$ denote the midpoint of $\gamma$, it follows
$$d(p,m)=d(q,m)=\frac{d(p,q)}{2}<\frac{\pi}{2}$$ so that $\{p,q\}\subset H_m$.  The hypothesis implies there exists $s \in (0,t_c) \cup (t_c,P_c)$ such that $c(s)\in E_m$.  After possibly reversing the orientation of $c$, we may assume that $s \in (0,t_c)$.  Let $z=c(s)$.   Then  \begin{equation}\label{alpha} d(p,z)\leq L(c,[0,s]) =\int_{t=0}^{t=s} \|c'(t)\|\,dt\end{equation} and \begin{equation}\label{beta}d(z,q)\leq L(c,[s,t_c])=  \int_{t=s}^{t=t_c} \|c'(t)\|\,dt.\end{equation} Summing, \begin{equation} \label{sum}d(p,z)+d(z,q)\leq L(c,[0,t_c])= \frac{L}{2}<\pi.\end{equation}   Consider the isometric reflection $F:\mathbb{S}^{n-1} \rightarrow \mathbb{S}^{n-1}$ about $E_m$ defined by $$F(x)=x-2\langle x,m\rangle m$$ for each $x \in \mathbb{S}^{n-1}.$  As $z \in E_m$, $F(z)=z$. Use $m=\frac{p+q}{\|p+q\|}$ to evaluate $F(-q)=p$. Now $$\pi=d(-q,q)\leq d(-q,z)+d(z,q)=d(F(-q),F(z))+d(z,q)=d(p,z)+d(z,q),$$  contradicting (\ref{sum}) and concluding the proof.     
\end{proof}

\begin{lemma}\label{length}
Let $C:\mathbb{R} \rightarrow \mathbb{R}^n$ be a $P$-periodic unit speed parameterization with $\|C''(t)\| \leq 1$ for each $t \in \mathbb{R}$.  Let $c:\mathbb{R} \rightarrow \mathbb{S}^{n-1}$ be the parameterization defined by $c(t)=C'(t)$ for each $t \in \mathbb{R}$.  Then the length of the closed curve $c(\mathbb{R})$ is  less than or equal to $P$ and equal to $P$ if and only if $c$ is a $P$-periodic unit speed parameterization.
\end{lemma}

\begin{proof}
Note that $c$ is not constant and that $P \in Per(c)$.  Therefore $c$ is a periodic parameterization with period $P_c$ satisfying $P_c \leq P$.  The length $L$ of the closed curve $c(\mathbb{R})$ satisfies $$L=\int_{0}^{P_c}\|c'(t)\|\,dt=\int_{0}^{P_c}\|C''(t)\|\,dt\leq \int_{0}^{P_c}1\,dt=P_c\leq P.$$ Note that $L=P$ if and only if $P_c=P$ and $\|c'(t)\|=\|C''(t)\|=1$ for each $t \in \mathbb{R}$, concluding the proof.
\end{proof}

\begin{proposition}\label{round}
Let $C:\mathbb{R} \rightarrow \mathbb{R}^n$ be a $2\pi$-periodic unit speed parameterization with $\|C''(t)\|\leq 1$ for each $t \in \mathbb{R}$.   Then $S=C(\mathbb{R})$ is a unit circle. 
\end{proposition}

\begin{proof}
Let $c:\mathbb{R} \rightarrow \mathbb{S}^{n-1}$ be the parameterization defined by $c(t)=C'(t)$ for each $t \in \mathbb{R}$.  By Lemmas \ref{derivnotinhemi}, \ref{long}, and \ref{length}, $c$ is a $2\pi$-periodic unit speed parameterization of a curve in $\mathbb{S}^{n-1}$ that intersects every equatorial subsphere.  Let $p=c(0)$ and $q=c(\pi)$.  

We first claim that $d(p,q)=\pi$.  If not, then $d(p,q)<\pi=\diam(\mathbb{S}^{n-1})$.  As in Lemma \ref{long}, let $m$ denote the midpoint of the unique minimizing geodesic joining $p$ to $q$. As above, $p$ and $q$ lie in $H_m$ and so up to changing the orientation of $c$, there exists $s \in (0,\pi)$ with the property that the point $z=c(s)$ satisfies $z \in E_m$.  Note that if we establish $d(p,z)+d(z,q)<\pi$, then the argument presented in Lemma \ref{long} applies to obtain a contradiction.  First note that $d(p,z)\leq s$ and $d(z,q) \leq \pi-s$ by (\ref{alpha}) and (\ref{beta}). If neither of these inequalities is strict, then the restrictions of $c$ to $[0,s]$ and to $[s,\pi]$ are unit speed geodesics connecting $p$ to $z$ and $z$ to $q$, respectively.  As $c$ is a unit speed parameterization these two geodesics meet smoothly at $z$.  Therefore the restriction of $c$ to $[0,\pi]$ is a geodesic of length $\pi$ connecting $p$ to $q$.  This contradicts the assumption $d(p,q)<\pi$ since all geodesics in $\mathbb{S}^{n-1}$ of length $\pi$ are minimizing.

Next, we claim that $c$ parameterizes a geodesic in $\mathbb{S}^{n-1}$.  Indeed, the restrictions of $c$ to $[0,\pi]$ and to $[\pi,2\pi]$ define two curves of length $\pi$ that connect the points $p$ and $q$.  As $d(p,q)=\pi$, these two curves are geodesics.  These two geodesics meet smoothly at both $p$ and $q$ since $c$ is a unit speed parameterization, concluding the proof that $c$ parameterizes a geodesic in $\mathbb{S}^{n-1}$.

Finally, we argue that $C$ parameterizes a unit circle.  Define $x:\mathbb{R} \rightarrow \mathbb{R}^n$ by $x(t)=(\cos(t),\sin(t),0,\ldots,0)$.  Then $x(\mathbb{R})=\mathbb{S}^1.$. As the isometry group of $\mathbb{S}^{n-1}$ acts transitively on unit tangent vectors, any two unit speed parameterized geodesics in $\mathbb{S}^{n-1}$ differ by an isometry of $\mathbb{S}^{n-1}$. Therefore, there exists an orthogonal matrix $A \in O(n)$ such that for each $t \in \mathbb{R}$, $$c(t)=x(t)\cdot A.$$  By the fundamental theorem,  for each $t \in \mathbb{R}$, $$C(t)-C(0)=\int_{s=0}^{s=t} c(s)\, ds=\int_{s=0}^{s=t} x(s) \cdot A \, ds.$$ Therefore, $$C(t)=(\sin(t),-\cos(t),0,\ldots,0)\cdot A +(C(0)+(0,1,0,\ldots, 0)\cdot A),$$ from which the Proposition follows.
\end{proof}

\subsection{Berger's sharp isoembolic inequality}
This subsection reviews the definitions of the conjugate and injectivity radii for a closed Riemannian manifold and states Berger's isoembolic inequality, a key tool in the proof of Theorem \ref{main theorem}.  We refer the reader to \cite[Chap. 13]{doca}  and \cite{berger} for more details.

Let $M$ be a $k$-dimensional connected and closed Riemannian manifold.  Given $p\in M$, let $T_pM$ and $S_pM$ denote the tangent space and unit tangent sphere of $M$ at the point $p$.  Let $$\exp_p:T_p M \rightarrow M$$ denote the exponential map at $p$.  The \textit{conjugate radius at $p$} is defined as the supremum of $r>0$ for which the restriction of $\exp_p$ to the ball $B(0,r)$ is nonsingular.  The \textit{conjugate radius of $M$}, denoted by $\conj(M)$, is defined as the infimum of the conjugate radii of its points.  

Given $p \in M$ and $v \in S_pM$, let $\gamma_v(t)=\exp_p(tv)$.  Then $\gamma_v:\mathbb{R} \rightarrow M$ is a unit speed parameterized geodesic.  The \textit{cut time of $(p,v)$} is the positive real number $c(p,v)$ defined by 

$$c(p,v)=\sup \{t>0\,\vert\, d(p,\gamma_v(t))=t\}.$$ The cut time defines a continuous function on the unit sphere bundle $$SM=\{(p,v)\,\vert\, p\in M\,\,\,\text{and}\,\,\,v\in S_pM\}.$$ The \textit{injectivity radius of M}, denoted $\inj(M)$, is defined by $$\inj(M)=\min\{c(p,v)\,\vert\, (p,v) \in SM\}.$$

The following is known as Klingenberg's injectivity radius estimate (see e.g. \cite[Chap. 13]{doca}).

\begin{lemma}[Klingenberg]\label{Klingenberglemma}
Let $l$ denote the length of a shortest nonconstant closed geodesic in $M$.  Then $\inj(M)=\min\{\conj(M),l/2\}.$

\end{lemma}

\begin{theorem}[Berger  \cite{berger} ]\label{bergerthm}
Let $M$ be a closed $k$-dimensional Riemannian manifold.    Then $$\vol(M) \geq \vol(\mathbb{S}^k) \left(\frac{\inj(M)}{\pi} \right)^k$$ with equality holding only for constant curvature spheres.
\end{theorem}

\subsection{Reach one submanifolds of \texorpdfstring{$\mathbb{R}^n$}{Rn}}
\begin{lemma}\label{reachproperties}
Let $M$ be a closed submanifold of $\mathbb{R}^n$ with $\tau(M)=1$.  Then
\begin{enumerate}
\item If $\gamma:\mathbb{R} \rightarrow \mathbb{R}^n$  is a unit speed parameterization of a geodesic in $M$, then for each $t \in \mathbb{R}$, $\|\gamma''(t)\|\leq 1$.
\item The sectional curvatures of $M$ are bounded above by $1$.
\item The injectivity radius of $M$ satisfies $\inj(M) \geq \pi$.
\end{enumerate}
\end{lemma}

\begin{proof}
Item (1) follows from the fact that the norm of the second fundamental form of $M$ is bounded above by $1$ in all normal directions \cite[Proposition 6.1]{NiSmWe}.  Item (2) follows from item (1) and the Gauss equation \cite[Chap. 6, Theorem 2.5]{doca}.  It remains to prove (3).  By Lemma \ref{Klingenberglemma}, the injectivity radius of $M$ equals the minimum of its conjugate radius and half the length of a shortest closed geodesic in $M$.  By Lemma \ref{reachproperties}-(2) and the Rauch comparison theorem \cite[Chap. 10, Theorem 2.3]{doca}, $\conj(M) \geq \pi$.  It remains to show the shortest closed geodesic in $M$ has length at least $2 \pi$.  Let $l>0$ denote the length of a shortest closed geodesic and let $C:\mathbb{R} \rightarrow M$ be an $l$-periodic unit speed parameterization of one such closed geodesic. Let $c=C'$ and let $L$ denote the length of the closed curve $c(\mathbb{R})$.  By Lemma \ref{reachproperties}-(1), Lemma \ref{length} applies, whence $$L\leq l.$$  By Lemmas \ref{derivnotinhemi} and \ref{long}, $$2\pi \leq L.$$ Therefore, $2\pi \leq l$, completing the proof. 
\end{proof}

\section{Proof of Theorem \ref{main theorem}.}\label{proof}  

\begin{proof}
By Theorem \ref{bergerthm} and Lemma \ref{reachproperties}-(3), $$\vol(M) \geq \vol(\mathbb{S}^k) \left(\frac{\inj(M)}{\pi} \right)^k \geq \vol(\mathbb{S}^k).$$  Now suppose that $\vol(M)=\vol(\mathbb{S}^k)$. Then $\inj(M)=\pi$ and $M$ is isometric to the canonical unit sphere $\mathbb{S}^k$.  In particular, each of its geodesics is a closed curve of length $2\pi$ admitting a unit speed parameterization.  By Lemma \ref{reachproperties}-(1) and Proposition \ref{round}, each geodesic in $M$ is a unit circle in $\mathbb{R}^n$.   By \cite[Theorem 4]{Ho}, $M$ and $\mathbb{S}^k$ are congruent.
\end{proof}

\textbf{Acknowledgement. } The authors would like to acknowledge the late Professor Christian Blatter for drawing their attention to the puzzle in \cite{Pe, Wi} through a delightful discussion on math stack exchange.

\bibliographystyle{abbrv}
\bibliography{refs}

\end{document}